\documentclass[12pt,a4paper]{amsart}
\usepackage{amsfonts}\usepackage{amssymb}\usepackage{amsmath}\usepackage{amsthm}
\usepackage[utf8]{inputenc}
\usepackage[T1]{fontenc}
\usepackage{latexsym}
\usepackage{graphicx}
\usepackage{layout}
\usepackage{subfigure}
\usepackage{pgfplots}
\usepackage[colorlinks=true,urlcolor=blue,citecolor=red,linkcolor=blue,linktocpage,pdfpagelabels,bookmarksnumbered,bookmarksopen]{hyperref}
\usepackage{tikz}
\usepackage{tikz-3dplot}
\usetikzlibrary{hobby}

\newtheorem{theorem}{Theorem}[section]

\newtheorem{proposition}[theorem]{Proposition}

\newtheorem{remark}[theorem]{Remark}

\numberwithin{equation}{section}

\def\R{\mathbb{R}}

\def\e{{\varepsilon}}
\def\di12{\mathcal{D}^{1,2}(\R^n)}

\def\l{{\lambda}}

\def\0l{_{0,\l}}
\def\1l{_{1,\l}}
\def\2l{_{2,\l}}
\def\3l{_{3,\l}}
\def\4l{_{4,\l}}

\def\Om{\Omega}

\def\beq{\begin{equation}}
\def\eeq{\end{equation}}
\def\sideremark#1{\ifvmode\leavevmode\fi\vadjust{\vbox to0pt{\vss
 \hbox to 0pt{\hskip\hsize\hskip1em
 \vbox{\hsize2.1cm\tiny\raggedright\pretolerance10000
  \noindent #1\hfill}\hss}\vbox to15pt{\vfil}\vss}}}%

\newtheorem*{theorem*}{Theorem}

\begin{document}
\title{ { Low regularity results for degenerate Poisson problems }}
\author[Calanchi]{Marta Calanchi} \address{Marta Calanchi, Dipartimento di Matematica, Universit\`a degli Studi   di Milano - Via  Saldini 50, 20133 Milano, Italy e-mail: {\sf marta.calanchi@unimi.it}.}
\author[Grossi]{Massimo Grossi }
\address{Massimo Grossi,  Dipartimento di Scienze di Base Applicate per l'Ingegneria, Universit\`a degli Studi di Roma \emph{La Sapienza} - Via Scarpa 10, 00161, Roma, e-mail: {\sf massimo.grossi@uniroma1.it}.}
\date{29 May 2025. \ \  Research project partially supported by INdAM-GNAMPA}
\begin{abstract}
In this paper we study the Poisson problem,
\[
\begin{cases}
-{\rm div}(d^\beta\nabla u)=f&{\rm in}\ \Om\\
u=0&{\rm on}\ \partial\Omega,
\end{cases}
\]
where  $\Om\subset\R^N$, $N\ge2$ is a smooth bounded domain, $f$ is a continuous function, $\beta< 1$, and $d(x)=dist(x,\partial\Omega )$. We  describe the behaviour of $u$ near $\partial\Om$ and discuss some of its regularity properties.

\end{abstract}

\maketitle
\section{Introduction}

\medskip

The study of partial differential equations  plays a fundamental role in mathematical analysis, with applications ranging from physics to engineering, and has profound implications in understanding the regularity of solutions in various settings. 

\medskip
In this work, we focus on a particular class of Poisson-type problems that involve  degenerate or singular elliptic operators. In many cases, these problems exhibit degeneracy at the boundary, where the behaviour of the solution becomes less regular or singular. 

\medskip
We consider the weighted   operator, denoted by $L_\beta$,  defined as
\begin{equation}\label{b1a}
L_\beta(u):=-{\rm div}(d^\beta\nabla u).
\end{equation}

where the weights $w(x)=d^\beta(x)$ with $\beta<1$ are powers of the distance function 
$$d(x)=dist(x,\partial\Omega ),$$ 
and 
 $\Omega\subset\mathbb R^N$ is a bounded domain with $\partial\Omega\in C^2$.

\medskip
We are interested principally in   the regularity but mainly in the precise behaviour at the boundary of {\it  classical } solutions of the following Poisson-type problem

\begin{equation}\label{b8}
\begin{cases}
L_\beta(u)=f&{\rm in}\ \Om\\
u=0&{\rm on}\ \partial\Omega,
\end{cases}
\end{equation}

\medskip

where $f$ is a given function that acts as a source term within the domain $\Omega$. 

Although the operator is degenerate, if $f\in C^{0,\alpha}(\Omega)$ for some $\alpha\in(0,1)$, the classical local Schauder estimates give at least  that $u\in C^2(\Omega)$, since $L_\beta$ is uniformly elliptic in every  compact set $\Omega'\subset \Omega$. 
  
    \medskip
The weight introduces a degeneracy near the boundary when $\beta>0$, or a singularity when $\beta<0$ into the operator, which affects the behaviour of the solutions, that in general become less regular. 

  \medskip

  We aim to refine the understanding of the regularity of solutions to \eqref{b8},  especially near the boundary of $\Omega$.

\medskip

 Since the operator is in divergence form, it is natural to also consider the problem  from a variational point of view, addressing both the existence and the  possible regularity of the {\it weak solutions} of (\ref{b8}). To this aim  it is helpful to introduce the natural (weighted) Sobolev space in which to frame the problem. We briefly recall these notions for a general weight $w\ge0$.

\medskip

  \medskip
For   $w\in L_{loc}^1(\Omega)$, $w(x)\ge 0 \ a.e.$, we denote with
$L^p(\Omega,
w )$ the weighted Lebesgue space of measurable functions $u$ on $\Omega$ with  finite norm
\begin{equation}
||u||_{L^p(\Omega,w)}=\left(\int_\Omega |u(x)|^pw(x)\ dx\right)^\frac1p
\end{equation}
and with $W^{k,p}(\Omega,w)$ the weighted Sobolev space of all measurable functions $u$ whose 
distributional derivatives belongs to $L^p(\Omega,w)$, i.e. for which the norm
\begin{equation}\label{W}
||u||_{W^{k,p}(\Omega,w)}=\left(\sum_{|\alpha|\le k}||D^\alpha u||^p_{L^p(\Omega,w)}\right)^\frac1p
\end{equation}
is finite.

\medskip

\medskip
It is well-established (see e.g \cite{ko}) that if $w^{-\frac 1p}\in L^{p'}_{loc}(\Omega)$, then   $W^{k,p}(\Omega,w)$ are Banach spaces, and  $C_0^\infty(\Omega)\subset W^{k,p}(\Omega,w)$. 
Therefore, one can  also define the space 
$$
 W_0^{k,p}(\Omega,w)=\overline{C_0^\infty(\Omega)}^{||\cdot||_{W^{k,p}(\Omega,w)}}
$$
as  the closure of $C^\infty_0(\Omega)$ with respect to the norm defined in (\ref{W}).

\medskip The literature offers several important results in this setting.
 An interesting reference for weighted Sobolev spaces and related inequalities   is the work by Edmunds and Opic, (\cite{eo}, see also the  work of Kufner and Opic \cite{ko}),
which 
 focuses on the weighted Poincar\'e inequality 
  treating in details 
 the special case of weights $w$  which are powers of the distance function $d(x).$  
 
 \medskip
{In this setting, 
we say that $u\in W^{1,2}_0(\Omega,w)$ is a weak solution to (\ref{b8}), if 
\begin{equation}\label{weak}
\int_\Omega \nabla u\nabla \phi \ w dx=\int_\Omega  f\phi \ dx,\quad\forall \phi\in C^\infty_0(\Omega).
\end{equation}

\medskip


\medskip

}

The present  work is partially inspired by the celebrated  paper of 
 Fabes, Kenig and Serapioni  (\cite{fks}), 
which investigates the local regularity properties of weak solutions to singular/degenerate elliptic partial differential equations.

\medskip

%
%
%
%

  Specifically, the cited paper focuses on the behaviour of weak solutions in regions where the ellipticity condition of the equation may fail or degenerate. The authors develop and apply methods to demonstrate that, under certain conditions,  solutions to these equations possess local and global regularity. 
  
  \medskip

  Their results involve nonnegative weights that belong to the Muckenhoupt classes $A_p$ (see e.g. \cite{Mk}), i.e.  weights $w$ defined as follows: for  $ \ p>1 $,   $w\in A_p$ if
  $$
\sup_{B\subset\Omega}\left(\frac1{|B|}{\int_Bw\  dx }\right)\left(\frac1{|B|}{\int_Bw^{-\frac{1}{p-1}}}dx \right)^{p-1}<+\infty.
  $$
 (For example the  weight $w=d^\beta\in A_2$ when  $\beta\in (-1,1)$, see Theorem 3.1 in \cite{dst}).

 \medskip
An important result in this direction is the following, which we state it in the case $w=d^\beta$, although it holds for more general weights as well:

\begin{theorem}[Theorem 2.4.8 in \cite{fks}] Let  $u\in W^{1,2}_0(\Omega, d^\beta)$ be a weak solution of 
$$
L_\beta(u)=-div \overrightarrow{F},
$$
where $\overrightarrow{F}:\Omega\to\mathbb R^N$ is a vector field such that $|\overrightarrow{F}|/d^\beta\in L^p(\Omega,d^\beta)$  with $\beta\in(-1,1)$ and $p>2n-\varepsilon$, for some $\epsilon>0$. Then, $u$ is H\"older continuous in $\overline\Omega$. 
\end{theorem}
%

{ From now on, we will assume that $w = d^\beta$, with $\beta < 1$. As in the result above, many works have been carried out within the framework of the Muckenhoupt class (see also \cite{fjk1,fkj2}); we also consider the case $\beta \le -1$.}

%
%

\medskip 
In this paper, we do not focus on identifying the optimal conditions on $f$
required to guarantee continuity or boundedness of the solution to \eqref{b8}. Rather, we take these regularity properties as given and use them as a basis to establish additional results. In particular, we derive explicit estimates that characterize the solution's asymptotic behaviour near the boundary.
These estimates provide upper and lower bounds for the solution $u$, highlighting the maximal regularity of the solution in terms of the parameters of the problem, in particular of  the exponent $\beta$.

  \medskip
The main result  in this direction is stated in our main theorem, which follows.

\medskip

\begin{theorem}\label{L3}
Suppose that $u\in C^2(\Omega)\cap C(\overline\Om)$ solves
\eqref{b8} with
$f\ge0$, $f\not\equiv 0$ in $\Om$, $f\in C(\overline\Om)$ and $\beta<1$.
Then,   for any $\eta_1,\eta_2>0$, there exist a sufficiently small $\sigma>0$ and two positive constants $D_1$ and $D_2$ such that
\begin{equation}\label{estimate2}
D_1d^{1-\beta}(x)(-\log d(x))^{-\eta_1}\le u(x)\le D_2d^{1-\beta}(x)(-\log d(x))^{\eta_2}, \quad \forall x\in\Gamma_\sigma^\circ.
\end{equation}
where $$\Gamma_\sigma=\{x\in\overline\Omega: \ d(x)<\sigma\}.
$$

Moreover, if $\Om$ is convex,  then
\begin{equation}\label{estimate3}
D_1d^{1-\beta}(x)(-\log d(x))^{-\eta_1}\le u(x)\le D_2d^{1-\beta}(x), \quad \forall x\in\Gamma_\sigma.
\end{equation}
\end{theorem}

\medskip
The estimate (\ref{estimate3}) provides a precise control of the solution 
$u$ on the boundary $\partial\Omega$.

\smallskip In our opinion it is interesting in itself, and in any case, we will use it to deduce some properties of the solution.
A first immediate consequence concerns the maximal H\"older regularity that we may expect from the solution. 

As mentioned before, 
in  \cite{fks}, under some assumptions on $f$, it was proved that the solution  $u\in C^{0,\alpha}(\overline\Om)$, for some $\alpha\in (0,1)$.  We give an upper bound to the value $\alpha$: 
\begin{theorem}\label{h} 
Suppose that $u\in C^2(\Omega)\cap C^{0,\alpha}(\overline\Om)$ { solves}
\eqref{b8} with
$f\ge0$, $f\not\equiv 0$ in $\Om$, $f\in C(\overline\Om)$ and $\beta<1$. 
Then  $\alpha\le\min\{1,1-\beta\}.$
\end{theorem}
We point out that the result is sharp, as proved in Remark \ref{sharp}.

\medskip

Another  consequence  concerns the regularity of this solution within the framework of Sobolev spaces.  The  solution to equation \eqref{b8} minimizes the functional 
 \begin{equation*}
 F(u):=\frac12 \int_\Omega|\nabla u|^2 d^\beta\ dx- \int_\Omega f u \ dx
 \end{equation*}
 in the space $W^{1,2}_0(\Om,d^\beta)$ (see Section \ref{s2}).
 The next result investigates the inclusion of 
the solution in the classical Sobolev spaces.

\begin{theorem}\label{t}

Same assumptions as in Theorem \ref{L3}. Then $u$ belongs to $W^{1,2}_0(\Omega)$ if and only if $\beta<\frac12$.
\end{theorem}

In the degenerate case ($\beta>0$),  Theorem \ref{t}  shows a different regularity of the solution depending on whether $\beta \in (0, \frac12)$ or $\beta \in [\frac12, 1)$, despite the fact that in both cases the weight $d^\beta$ belongs to the Muckenhoupt class $A_2$. Therefore, mere belonging in the $A_2$ class is not sufficient to provide an accurate description of the regularity of the solution. The same remark applies to Theorem \ref{h}.

 \medskip
 
 Lastly, we note that our results extend to more general weights 
$
w$ that behave similarly 
near the boundary (see Remark \ref{r1}).

\medskip
{
Some other regularity results have been proven in \cite{stv1} and \cite{stv2}  for the weight $w(x_1,..,x_N)=|x_1|^\beta$ with $\beta\in\R$ and $x\in B_1$, the unit ball of $\R^N$.
In particular, in \cite{stv2} it is developed a more general and structured theory to address the regularity of odd solutions whereas our paper is concerned with providing more detailed results in the case where the weight $w$ is the distance from the boundary. Another difference is methodological: in \cite{stv2}, non-degenerate approximating problems are considered, whereas in our case we deal directly with the operator.

\medskip
We also mention the papers \cite{fjk1,fkj2} for other properties of the solution including the discussion of regular point of $\partial\Omega$.}

\medskip

The paper is organized as follows:  in Section \ref{s2} we provide some additional properties of the weight  $d^\beta$ and recall other known results;  in Section \ref{s3} we give the proof of  the main estimates of Theorem \ref{L3},  and in Section \ref{s4} we prove Theorems \ref{h} and \ref{t}.

\bigskip

{\it Acknowledgements.} The authors wish to express the gratitude to Sergio Polidoro for the valuable suggestion that significantly improve the formulation of Theorem \ref{t}.

\section{Notations and preliminary known results}\label{s2}

\medskip

%

\medskip 
Before presenting the main result, we would like to highlight some key properties of $d$.

\medskip
We denote by
\begin{equation}\label{a1}
\Gamma_\sigma =\{x\in\overline\Omega: \ d(x)<\sigma\}
\end{equation}
the portion in $\overline\Omega$ of a tubular neighbourhood of $\partial\Omega$. With an abuse of terminology, 
from now on we will call $\Gamma_\sigma$ a neighbourhood of $\partial\Omega$.

\begin{proposition}
Let $\Omega \subset\mathbb R^N$ a bounded domain with $\partial \Omega\in C^2$. Then  there exists a small constant $\sigma>0$ such that
\medskip
\begin{equation}\label{d1}
d\in C^2(\Gamma_\sigma\cap\Omega)\cap C^0(\overline\Gamma_\sigma),
\end{equation}
\begin{equation}\label{d2}
|\nabla d(x)|=1\hbox{ for all }x\in\Gamma_\sigma,
\end{equation}
Moreover, 
for every measurable nonnegative function $g: (0,\sigma)\to \mathbb R$
\begin{equation}\label{b11}
g\circ d \in L^1({\Gamma_\sigma})\Longleftrightarrow g\in L^1(0,\sigma).
\end{equation}

\end{proposition}

\medskip

\begin{proof} For (\ref{d1}) and  (\ref{d2}) see e.g. \cite{gt} Appendix 14.6.

\medskip
Here's a brief outline of how to prove \eqref{b11}:
from the coarea formula and  (\ref{d1}) and  (\ref{d2}),  we have, since $|\nabla d(x)|=1$

\begin{equation*}
\int_{\Gamma_\sigma}g(d(x))\ dx 
=\int_0^\sigma  g(t)\mathcal H^{N-1}(\Gamma_\sigma\cap\{d=t\})\ dt,
\end{equation*}
where $ \mathcal H^{N-1}$ is the Hausdorff measure of $\Gamma_\sigma\cap\{d=t\}$. Since the $\partial \Omega$ is $C^2$, there exist two positive constants $c_1$ and $c_2$ such that 
$$
c_1\mathcal H^{N-1}(\partial \Omega)\le \mathcal H^{N-1}(\Gamma_\sigma\cap\{d=t\})\le c_1\mathcal H^{N-1}(\partial \Omega)
$$
(see e.g. \cite{sl} , Appendix  2.12.3). This ends the proof.
\end{proof}

Let us end this section with some remarks on the weak solution to \eqref{b8}.
If $f/d^\beta\in L^2(\Omega,d^\beta)$, 
then \eqref{b8} admits a weak solution $u\in W^{1,2}_0(\Om,d^\beta)$. This is an immediate consequence of weighted Poincar\'e inequality

\begin{equation}\label{poincare}
\int_\Omega |u(x)|^2d^\beta(x)\ dx\le C\int_\Omega|\nabla u|^2 d^\beta(x)\ dx,
\end{equation}
and the compact embedding of $W^{1,2}_0(\Om,d^\beta)$ in $L^2(\Omega,d^\beta)$
 (see e.g \cite{eo}, Prop. 5.1 and Ex. 5.2).
 
 Indeed, the functional 
  $F:W^{1,2}_0(\Om,d^\beta)\to \mathbb R$ 
 \begin{equation}\label{functional}
 F(u):=\frac12 \int_\Omega|\nabla u|^2 d^\beta\ dx- \int_\Omega f u \ dx
 \end{equation}
 is well defined, coercive and bounded from below. So that $J$  attains its infimum on $W^{1,2}_0(\Omega,d^\beta)$ (see also  Theorem 2.9 in \cite{cav2}). 
 \medskip

\vspace{0.5cm}

\section{The main estimate: proof of Theorem \ref{L3}}\label{s3}

\medskip
In this section, we will prove upper and lower bounds for the solutions of the problem (\ref{b8}) where the datum
$f$ is a non-negative continuous function.

\medskip

Let us give the proof of Theorem \ref{L3},

\medskip
\begin{proof}[Proof of Theorem \ref{L3}] We begin by noting  that, since $L_\beta$ is elliptic in $\Om$, the weak maximum principle applies (see \cite{gt}). Therefore,  $f\ge0$ implies that $u\ge0$. Furthermore, since $L_\beta$ is uniformly elliptic in every $\Om'\subset\subset\Om$, 
 it follows that $u>0$ in $\Omega$.
Now, we split the proof into several steps.

\medskip
Let, for $\eta\in\R$ and $x\in \Gamma_\sigma$,

\begin{equation}\label{v}
v(x)=\begin{cases}
d^{1-\beta}(x)\left(-\log d(x)\right)^\eta\quad &x \in\ \Gamma_\sigma\setminus \partial\Om\\
0&x\in\ \partial\Omega,
\end{cases}
\end{equation}

{\bf Step 1:} {we have
\begin{equation}\label{Lv}
  \begin{split}
L_\beta( v)&=-\Delta d\left[(1-\beta)\left(-\log d\right)^\eta-\eta\left(-\log d\right)^{\eta-1}\right]\\
&+d^{-1}\left[
(1-\beta)\eta\left(-\log d\right)^{\eta-1}
-\eta(\eta-1)\left(-\log d\right)^{\eta-2}\right].
\end{split} 
\end{equation}
}
This is a straightforward computation where we used that $|\nabla d|=1$.

\medskip
{\bf Step 2:} If $\eta<0$ we have that here exists $\sigma>0$  and $\epsilon_1>0$ such that 
\begin{equation}\label{u>v}
u(x)\ge \epsilon_1 d^{1-\beta}(x)\left(-\log d(x)\right)^\eta, \quad\forall x\in  \Gamma_\sigma.
\end{equation}

\medskip
Indeed, from (\ref{Lv}) we have that,  (recall that  $\Delta d$  is bounded  in a neighbourhood  $\Gamma_\sigma$ of  $\partial\Omega$)
$$
L_\beta(v)=\frac{(1-\beta)\eta}{d\left(-\log d\right)^{1-\eta}}+o\left(\frac{1}{d\left(-\log d\right)^{1-\eta}}\right)\quad {\rm as } \ x\to \partial \Omega.
$$
Since $L_\beta(u)=f\ge0$ and $\eta<0$,  there exists  a  (possibly smaller) $\sigma>0$ such that
\begin{equation}\label{bb15}
L_\beta(v)<0\le L_\beta(u), \quad \forall x\in\Gamma_\sigma.
\end{equation}

The claim will follow by the maximum principle (see Theorem 3.1 in \cite{gt}). Indeed,  for a small $\epsilon_1$, it holds
\begin{equation}\label{bb6}
\epsilon_1 v(x)\le   u(x),\quad\forall x\in\partial\Gamma_\sigma.
\end{equation}
This is obvious on $\partial\Om$
because both $u$ and $v$ vanish on $\partial\Om$ and, on the other hand, since $u$ is continuous in $\Om$ and strictly positive  on $\partial \Gamma_\sigma\cap\Om$,  we have that, possibly choosing a small $\epsilon_1>0$, 
\begin{equation*}
\epsilon_1\max\limits_{x\in \partial \Gamma_\sigma\cap\Om}v\le \min\limits_{x\in  \partial \Gamma_\sigma\cap\Om}u.
\end{equation*}
Resuming,  from \eqref{bb15} and \eqref{bb6} we derive that 
\begin{equation*}
L_\beta(u-\epsilon_1v)\ge 0 {\ \rm on}\ \Gamma_\sigma,\  {\rm and}\ \ (u-\epsilon_1 v)\ge0 \  {\ \rm on}\ \  \partial\Gamma_\sigma.
\end{equation*}
Consequently,  by the weak maximum principle
\eqref{u>v} holds,
and  the claim of  Step 2 is proved.

\bigskip

%
%
%
{\bf Step 3: } If $\eta>0$ we have that here exists $\Gamma_\sigma$  and $\epsilon_2>0$ such that
\begin{equation}\label{u<v}
 \epsilon_2 u(x)\le d^{1-\beta}(x)\left(-\log d(x)\right)^\eta, \quad\forall x\in \overline \Gamma_\sigma.
\end{equation}

\bigskip
Again, from (\ref{Lv}) we have that
$$
L_\beta(v)=\frac{(1-\beta)\eta}{d\left(-\log d\right)^{1-\eta}}+o\left(\frac{1}{d\left(-\log d\right)^{1-\eta}}\right)\quad {\rm as } \ x\to \partial \Omega.
$$
Therefore, since $\eta>0$, for every $M>0$ there exists a neighbourhood $\Gamma_\sigma$ of $\partial \Omega$ such that
\begin{equation}\label{bb13}
L_\beta(v)>M \quad \forall x\in\Gamma_\sigma.
\end{equation}
Now, by the continuity of $f$ on $\overline \Omega$, we have that 
$$
0<\max_{x\in\overline\Omega}f(x)=\max_{x\in\overline\Omega}L_\beta(u)
 <+\infty
$$

We choose $\displaystyle M\ge \max_{x\in\overline\Omega}f(x) $,  so that 
\begin{equation}
L_\beta(v)\ge L_\beta(u).
\end{equation}

\medskip
With an argument  similar to the one used in the previous step, by reversing the roles of $u$ and $v$, we also have that there exists a constant {$\e_2>0$} such that 

\begin{equation}
\min_{\partial \Gamma_\sigma}v \ge \max_{\partial \Gamma_\sigma}\e_2 u,
\end{equation}
and the claim follows by the weak maximum principle.

\vspace{1cm}
{\bf Step 4: } Proof of \eqref{estimate2}.

\bigskip
It follows by Step $2$ and Step $3$ with $D_1=\e_1^{-1}$ and $D_2=\e_2^{-1}$.

\bigskip
{\bf Step 5: } Proof of \eqref{estimate3}. 
Here we have only to prove the RHS of \eqref{estimate3}. The convexity of $\Om$ will allow to choose $\eta=0$ in \eqref{Lv}
Indeed, if $\eta=0$ then \eqref{Lv} becomes
$$L_\beta(v)=-(1-\beta)\Delta d\ge0$$
since $\Omega$ is convex (see Lemma 14.17 in \cite{gt}). Next the claim of the RHS of \eqref{estimate3}  follows as in the proof of Step $3$.
\end{proof}

\begin{remark}\label{r1}
The previous proposition can be easily adapted  to more general weights. Indeed, the same statement holds if the distance function is replaced by a $C^2$-function, $w>0$ in a `neighborhood' of $\partial\Om$ and such that $w$ has no critical point on $\partial\Om$. This last claim implies that there exists a positive constant $C$ such that
$$\frac1C\le|\nabla w|^2\le C\quad\hbox{in a neighborhood of }\partial\Om.
$$
and this allows to repeat the proof without any change.
\end{remark}

\vspace{0.5cm}
\section{Proof of Theorems \ref{h} and \ref{t}}\label{s4}

\medskip

This section is devoted to derive additional  properties  for solutions to \eqref{b8}, based on the estimates established in Theorem \ref{L3}.


\begin{proof}[Proof of Theorem \ref{h}] Note that the case $\beta\le 0 $ is trivial, so  we assume $\beta\in (0,1)$.
\smallskip
By contradiction suppose that $u\in C^{0,\alpha}(\overline\Om)$ with $\alpha>1-\beta$, i.e. there exists $C>0$ such that
\begin{equation}\label{b12}
|u(x)-u(y)|\le C|x-y|^\alpha\quad\forall x,y\in\overline\Om.
\end{equation}
In particular, let us choose $x\in \Gamma_\sigma$ and $y\in\partial\Om$ such that $|x-y|=d(x)$. Here $\Gamma_\sigma$ is the set where \eqref{estimate2} holds. Hence \eqref{b12} becomes, for $x\in \Gamma_\sigma$,
\begin{equation}
u(x)\le Cd^\alpha(x),
\end{equation}
and by the LHS of \eqref{estimate2}, we deduce that
\begin{equation*}
D_1d^{1-\beta}(x)(-\log d(x))^{-\eta_1}\le Cd^\alpha(x),\end{equation*}
which implies, since $\alpha>1-\beta$,
\begin{equation*}
 \frac C{D_1}\le d^{\alpha+\beta-1}(x)(-\log d(x))^{\eta_1}\to 0, \ {\rm as}\ x\to \partial\Omega,
\end{equation*}
 which leads to a contradiction.
\end{proof}
\begin{remark}\label{sharp}
The result of Theorem \ref{h} is sharp, as shown by the following two examples.
\end{remark}
\medskip 
Let $B_1$ denote the  unit ball  in $\R^N$, centered at the origin. 
For $\beta\in [0,1)$, we consider the  function 
$$u(x)=\int_{|x|}^1s^{1-N}(1-s)^{-\beta}\left(\int_0^st^{N-1}(1-t)^{\beta}dt\right)ds, \quad x\in B_1.$$
We have that $u$ is a positive radial solution to 
\begin{equation}
\begin{cases}
L_\beta(u)=d^\beta&in\ B_1\\
u=0&on\ \partial B_1,
\end{cases}
\end{equation}
 straightforward computation shows that shows that $u$ is H\"older continuous and satisfies $u(x)\sim(1-|x|)^{1-\beta}$ as  $|x|\to 1$. Therefore,  the optimal regularity of $u$ near the boundary  is $C^{0,1-\beta}$.

\medskip Alternatively,  for every $\beta<1$, we can consider
$$u(x)=\frac{\big(1+(1-\beta)|x|)}{N(1-\beta)(2-\beta)}(1-|x|)^{1-\beta}, \quad x\in B_1. $$
The function  $u$ is a positive radial solution to 
\begin{equation}
\begin{cases}
L_\beta(u)=1&in\ B_1\\
u=0&on\ \partial B_1,
\end{cases}
\end{equation}
Even in this case,   the maximum regularity of $u$ near the boundary  is $C^{0,1-\beta}$.

\medskip
Furthermore, the solution in this second example exhibits increasing regularity 
as $\beta \to -\infty$; in particular, $u \in C^k$ in a neighbourhood of the boundary if $\beta \le 1 - k$. This suggests a general phenomenon that merits further investigation.

\bigskip

\bigskip
A second intriguing question concerns the regularity of this solution in the context of  Sobolev spaces. 
We have previously established that  solutions to \eqref{b8} minimize the functional 
$$    F:W^{1,2}_0(\Omega,d^\beta)\to\R,\quad\quad F(u)=\int_\Om|\nabla u|^2d^\beta \ dx -\int_\Om fu \ dx.$$
Now we prove that $u\in W^{1,2}_0(\Omega) $ if and only if $\beta<\frac12$. 

\begin{proof}[Proof of Theorem \ref{t}] 
 The case $\beta\le0$ is trivial, so we consider only the case $\beta> 0$. 

\medskip We first prove that the  condition $\beta<\frac12$ is necessary. 

\medskip
Indeed, if by contradiction  $u\in W^{1,2}_0(\Omega)$,  and $\beta\ge\frac12$,  by the Hardy inequality and \eqref{estimate2},  we get
\begin{equation*}
\int_\Omega|\nabla u|^2\ dx\ge C \int_\Omega\frac{u^2}{d^2} \ dx
\ge C \int_{\Gamma_\sigma} \frac{D_1^2}{d^{2\beta}\left(-\log d\right)^{2\eta_1}}\ dx=+\infty,  
\end{equation*}
by  using  \eqref{b11} in Section \ref{s2}.

\medskip
Now, let $\beta<\frac12$. Starting from the equation
\[
-{\rm div}(d^\beta \nabla u) = f,
\]
we multiply both sides by $\displaystyle \frac{u}{d^\beta}$ and integrate over the set
\[
\Omega_c := \{x \in \Omega : u(x) > c\},
\]
\begin{equation}\label{a5}
-\int_{\Om_{c}} {\rm div}(\nabla ud^\beta)\frac u{d^\beta}\ dx=\int_{\Om_{c}} \frac{fu}{d^\beta}\ dx.
\end{equation}

\smallskip

Although the classic divergence theorem is not directly applicable (since $\partial\Omega_c$ is not guaranteed to be regular), thanks to Theorems 3.2 and 3.3 in \cite{mpp}, we have that 
\begin{equation}\label{marta}
\int_{\Om_c} {\rm div}(u\nabla u)\ dx =\int_{\partial{\Om_{c}}}\frac{\partial u}{\partial\nu}uds(x),
\end{equation}
holds for almost every $c\in\R$. So we have
\begin{equation}\label{m1}
\begin{split}
-\int_{\Om_{c}} {\rm div}(\nabla ud^\beta)\frac u{d^\beta}\ dx&=\underbrace{-\int_{\partial{\Om_{c}}}\frac{\partial u}{\partial\nu}uds(x)}_{\ge0}+\int_{\Om_{c}}|\nabla u|^2\ dx-\beta\int_{\Om_{c}}(\nabla u\cdot\nabla d)\frac u{d}\ dx\\
&
\ge\int_{\Om_{c}}|\nabla u|^2\ dx-|\beta|\int_{\Om_{c}}|\nabla u|\frac ud \ dx\\
&\ge \int_{\Om_{c}}|\nabla u|^2\ dx-|\beta|\frac\epsilon2\int_{\Om_{c}}|\nabla u|^2\ dx-\frac{|\beta|}{2\epsilon}\int_{\Om_{c}}\frac{u^2}{d^2} \ dx,
\end{split}
\end{equation}
where we used that   $|\nabla d|=1$, $\frac{\partial u}{\partial\nu}\le0$ on $\partial\Om_{c}$ and  the inequality  $\displaystyle ab\le\frac\epsilon2 a^2+\frac1{2\epsilon} b^2$ with $a=|\nabla u | $ and $b=\frac ud$. 

\medskip Now,  by choosing
$\epsilon=\frac1{|\beta|}$ in the last inequality, and by \eqref{a5},  one obtains
\begin{equation}\label{m2}
\int_{\Om_{c}} \frac{fu}{d^\beta}\ dx=-\int_{\Om_{c}} {\rm div}(\nabla ud^\beta)\frac u{d^\beta}\ dx
\ge\frac12\int_{\Om_{c}}|\nabla u|^2 \ dx-\frac{\beta^2}2\int_{\Om_{c}}\frac{u^2}{d^2}\ dx.
\end{equation}

\medskip
Therefore, 
\begin{equation*}
\int_{\Om_{c}}|\nabla u|^2 \ dx\le {\beta^2}\int_{\Om_{c}}\frac{u^2}{d^2}\ dx +2\int_{\Om_{c}} \frac {fu}{d^\beta}\ dx\le {\beta^2}\int_\Om\frac{u^2}{d^2}\ dx +2\int_\Om\frac {fu}{d^\beta}\ dx
\end{equation*}
We only need to estimate  the first integral on the right side, since the second is bounded by \eqref{b11} and by the continuity of $u$ and $f$. 

\medskip

In fact, by \eqref{estimate3}, for a fixed small $\sigma>0$ as in Theorem \ref{L3}, 
\begin{align*} \int_\Om\frac{u^2}{d^2}\ dx&=\int_{\Gamma_{\sigma}\cap\Om}\frac{u^2}{d^2}\ dx+\int_{\Om\setminus\Gamma_{\sigma}}\frac{u^2}{d^2}\ dx\\
&\le D \int_{\Gamma_{\sigma}\cap\Omega}\frac{\left(-\log d\right)^{2\eta}}{d^{2\beta}} \ dx+\frac{1}{\sigma^2}\int_{\Om\setminus\Gamma_{\sigma}} u^2 \ dx <+\infty, 
\end{align*}
since $\beta<\frac12$ using again  \eqref{b11} in Section \ref{s2}. \\
We now select  a  monotone decreasing sequence $c_n\to 0$  such that  \eqref{marta} holds; applying Beppo-Levi's Monotone Convergence Theorem we obtain
$$\int_\Om|\nabla u|^2    \ dx=\lim_{n\to +\infty}\int_{\Om_{c_n}}|\nabla u|^2 \ dx\le C.$$
This concludes the proof of the sufficient condition.
\end{proof}

It is worth noting that the threshold $\beta=\frac12$ is independent  of the regularity of the right-hand side  $f$. In particular, even if the datum $f$ in \eqref{b8} belongs to $C^\infty(\Om)$, the solution does not belong to $W^{1,2}_0(\Om)$ for $\beta\ge\frac12$. This result highlights that the regularity in Sobolev spaces 
is primarily dictated by the properties of the operator 
 $L_\beta$ rather than by the smoothness of the forcing term $f$. 

\vspace{2cm}

\bibliography{CalanchiGrossi2.bib}
\bibliographystyle{abbrv}

\end{document}